\documentclass[a4paper,11pt]{article}

\usepackage[square,numbers]{natbib}
\usepackage{amsfonts,amsmath,amssymb,amsthm}
\usepackage{hyperref}
\usepackage{graphicx}
\usepackage{mathpazo}	
\usepackage{tikz}

\theoremstyle{definition}
\newtheorem{theorem}{Theorem}[section]
\newtheorem{lemma}[theorem]{Lemma}
\newtheorem{prop}[theorem]{Proposition}

\newtheorem{rem}[theorem]{Remark}
\newtheorem{ex}[theorem]{Example}

\title{On the approximate periodicity of sequences attached to noncrystallographic root systems}
\author{Philipp Lampe}

\begin{document}
\maketitle

\begin{abstract}
We study Fomin-Zelevinsky's mutation rule in the context of noncrystallographic root systems. In particular, we construct approximately periodic sequences of real numbers for the noncrystallographic root systems of rank 2 by adjusting the exchange relation for cluster algebras. Moreover, we describe matrix mutation classes for type $H_3$ and $H_4$. 
\end{abstract}

\section{Introduction}
Fomin and Zelevinsky have introduced cluster algebras in an impactful article \cite[Definition 2.3]{FZ}. In the last ten years diverse authors have found cluster algebra structures in various branches of mathematics such as representation theory, algebra and combinatorics. To define a cluster algebra, Fomin-Zelevinsky have defined \textit{seeds} and \textit{mutations of seeds}. Here, a seed (without frozen variables) is a pair $(\mathbf{x},B)$ which consists of a \textit{cluster} and a \textit{mutation matrix}. The cluster $\mathbf{x}=(x_1,x_2,\ldots,x_n)$ is a sequence of \textit{cluster variables} and the mutation matrix $B$ is a skew-symmetrizable integer $n\times n$ matrix. Given an initial seed, the cluster algebra is now defined to be generated by all cluster variables in all seeds that are obtained from the initial seed by a sequence of mutations. The natural number $n$ is called the \textit{rank} of the cluster algebra.

Some cluster algebras are of \textit{finite type} and some cluster algebras are of \textit{infinite type}. Here, we say a cluster algebra is of finite type if the mutation process yields only finitely many cluster variables. In another impactful article, Fomin-Zelevinsky \cite[Theorem 1.4]{FZ2} have classified the cluster algebras of finite type via finite type root systems. The theorem implies that finite type cluster algebras (without frozen variables) are in bijection with \textit{Dynkin diagrams} of type $A_n (n \geq 1)$, $B_n (n\geq 2)$, $C_n (n\geq 3)$, $D_n (n\geq 4)$, $E_n (n=6,7,8)$, $F_4$, and $G_2$.  

These Dynkin diagrams classify crystallographic root systems. In particular, such a diagram visualizes the Coxeter structure of the Weyl group of the corresponding root system. In the setup of cluster algebras, the crystallographic condition yields integer entries in the mutation matrix $B$. On the other hand, finite Coxeter groups are in bijection with \textit{Coxeter-Dynkin} diagrams. Coxeter-Dynkin diagrams do not necessarily satisfy the crystallographic condition. Examples of non crystallographic Coxeter groups are dihedral groups (with Coxeter-Dynkin diagram $I_2(m)$ with $m=5$ or $m\geq 7$) and the symmetry group of the icosahedron (with Coxeter-Dynkin diagram $H_3$). 

The aim of this note is to generalize Fomin-Zelevinsky's matrix and seed mutation to noncrystallographic root systems. For every noncrystallographic root system of rank $2$ the mutation class of the $B$-matrix contains two elements. Given two initial positive real numbers we define a sequence of real numbers by adjusting the exchange relations for cluster algebras to our setup. It turns out that the sequence is no longer a periodic sequence, but it is an almost periodic sequence meaning that it is approximately equal to a periodic sequence. For the noncrystallographic root system of type $H_3$ the mutation class of the $B$-matrix is finite, but we do not observe the phenomenon of almost periodicity. The question of approximate periodicity in this setup has also been touched by Reading-Speyer, see Armstrong \cite[Problem 6.4]{A}.

\section{Background}

\subsection{Fomin-Zelevinsky's cluster algebras}

In this section we wish to recall the definition of Fomin-Zelevinsky's cluster algebras. We only consider coefficient-free cluster algebras without frozen variables over the field of rational numbers. 

Let $n\geq 1$ be an integer and let $u_1,u_2,\ldots,u_n$ be algebraically independent variables over the field $\mathbb{Q}$ of rational numbers. The field $\mathcal{F}=\mathbb{Q}(u_1,u_2,\ldots,u_n)$ of rational functions is also called the \textit{ambient field}. A \textit{cluster} is a sequence $\mathbf{x}=(x_1,x_2,\ldots,x_n)\in\mathcal{F}^n$ of algebraically independent elements. An $n\times n$ matrix $B=(b_{ij})$ with integer entries is called \textit{skew-symmetrizable} if there exists a diagonal $n\times n$ matrix $D=\textrm{diag}(d_1,d_2,\ldots,d_n)$ with positive integer diagonal entries such that the matrix $DB$ is skew-symmetric, i.\,e. $d_ib_{ij}=-d_jb_{ji}$ for all $1\leq i,j\leq n$. In this case, the diagonal matrix $D$ is called a \textit{skew-symmetrizer} of $B$. A \textit{seed} is a pair $(\mathbf{x},B)$ formed by a cluster $\mathbf{x}$ and a skew-symmetrizable integer $n\times n$ matrix $B$. We denote the set of seeds by $\mathcal{S}$.

Let $k\in\{1,2,\ldots,n\}$ be a natural number. A \textit{mutation} in direction $k$ is a map $\mu_k\colon \mathcal{S}\to\mathcal{S},(\mathbf{x},B)\mapsto\mu_k(\mathbf{x},B)=(\mathbf{x'},B')$, where $\mathbf{x'}=(x_1',x_2',\ldots,x_n')$ is the sequence that we obtain from the cluster $\mathbf{x}$ by replacing the variable $x_k$ with
\begin{align*}
x_k'=\frac{1}{x_k}\left(\prod_{b_{ik}>0}x_k^{b_{ik}}+\prod_{b_{ik}<0}x_k^{-b_{ik}}\right)\in\mathcal{F},
\end{align*}
and keeping all other cluster variables $x_i'=x_i$ with $i\neq k$, and $B'=(b_{ij}')$ is the $n\times n$ matrix with
\begin{align*}
b'_{ij}=\begin{cases}
-{b_{ij}},&\textrm{if } k\in\{i,j\};\\
{b_{ij}}+\frac{\vert b_{ik}\vert b_{kj}+b_{ik}\vert b_{kj}\vert}2,&\textrm{otherwise}.
\end{cases}
\end{align*} 
For every seed $(\mathbf{x},B)\in\mathcal{S}$ the pair $\mu_k(\mathbf{x},B)=(\mathbf{x'},B')$ is again seed, i.\,e. the elements of the sequence $\mathbf{x'}$ are also algebraically independent over the field of rational numbers and the matrix $B'$ has integer entries and is skew-symmetrizable (with the same skew-symmetrizer $D$). Thus the map $\mu_k\colon \mathcal{S}\to\mathcal{S}$ is well-defined.

Fomin-Zelevinsky's mutation of seeds has many remarkable properties. Firstly, for every index $k$ we have $\mu_k^2=\operatorname{id}_{\mathcal{S}}$ so that the map $\mu_k$ is an involution. We declare two seeds $(\mathbf{x},B),(\mathbf{x'},B')\in\mathcal{S}$ to be \textit{mutation equivalent} if there exists a sequence $(k_1,k_2,\ldots,k_r)$ of indices such that $(\mathbf{x},B)=(\mu_{k_1}\circ\mu_{k_2}\circ\ldots\circ\mu_{k_r})(\mathbf{x'},B')$. In this case we write $(\mathbf{x},B)\simeq(\mathbf{x'},B')$. It follows that $\simeq$ is an equivalence relation on the set of all seeds.

Suppose that $(\mathbf{x},B)$ is an initial seed. The cluster algebra $\mathcal{A}(\mathbf{x},B)\subseteq\mathcal{F}$ is the $\mathbb{Q}$-subalgebra generated by all cluster variables $x_k'$ in all seeds $(\mathbf{x'},B')$ that are mutation equivalent to $(\mathbf{x},B)$. By construction we have $\mathcal{A}(\mathbf{x},B)\subseteq \mathcal{F}$. More generally, Fomin-Zelevinsky's \textit{Laurent phenomenon} \cite[Theorem 3.1]{FZ} asserts that $\mathcal{A}(\mathbf{x},B)\subseteq \mathbb{Q}[x_1^{\pm 1},x_2^{\pm 1},\ldots,x_n^{\pm 1}]$. More generally, every cluster variable is an element in the ring $\mathbb{Z}[x_1^{\pm 1},x_2^{\pm 1},\ldots,x_n^{\pm 1}]$.

\subsection{Cluster algebras of rank $2$}

A skew-symmetrizable integer $2\times 2$ matrix has the form $B=\pm\left(\begin{smallmatrix}0&a\\-b&0\end{smallmatrix}\right)$ for some natural numbers $a,b\geq 1$. Note that the two possible choices of the sign yield isomorphic cluster algebras which we will denote by $\mathcal{A}(a,b)$. We can parametrize the cluster variables in $\mathcal{A}(a,b)$ by the set of integers, so that we obtain cluster variables $x_i$, with $i\in\mathbb{Z}$, and clusters $(x_{i-1},x_i)$, with $i\in\mathbb{Z}$. The equation
\begin{align*}
x_{i-1}x_{i+1}=\begin{cases}x_i^a+1,&\textrm{if \ } i \textrm{ \ is even;}\\x_i^b+1,&\textrm{if \ } i \textrm{ \ is odd;}\end{cases}
\end{align*}
describes the mutation from the cluster $(x_{i-1},x_{i})$ to the cluster $(x_{i},x_{i+1})$. Fomin-Zelevinsky's classification theorem implies that the cluster algebra $\mathcal{A}(a,b)$ is of finite type if and only if $ab<4$. In these cases, the sequence $(x_i)_{i\in\mathbb{Z}}$ is a periodic sequence. The period of the sequence is equal to $5$, $6$, or $8$ when $(a,b)$ is equal to $(1,1)$, $(2,1)$, or $(3,1)$.

Assume that $\mathcal{A}(a,b)$ is of finite type and let $i\in\mathbb{Z}$ be an integer. Due to the Laurent phenomenon we can write $x_i=f(x_1,x_2)/(x_1^{a_1}x_2^{a_2})$ for some polynomial $f\in\mathbb{Z}[x_1,x_2]$. It is easy to see that in these cases the constant term in the polynomial $f(x_1,x_2)$ is always equal to $1$. Evaluation of the Laurent polynomial at the pair $(x_1,x_2)$ yields a function $x_i\colon \mathbb{R}^+\times\mathbb{R}^+\to \mathbb{R}^+$. Motivated from tropical geometry, we could approximate $x_i\approx 1/(x_1^{a_1}x_2^{a_2})$. In this paper we wish to introduce an approximation, which is more accurate than the tropical approximation and also works in a more general (noncrystallographic) setup where we do not have a Laurent phenomenon.

\section{Approximately periodic sequences attached to noncrystallographic root systems of rank 2}

\subsection{The definition of the almost periodic sequences for type I}

The Dynkin diagrams attached to the cluster algebras $\mathcal{A}(1,1)$, $\mathcal{A}(2,1)$ and $\mathcal{A}(3,1)$ are $A_2$, $B_2$ and $G_2$, respectively. The corresponding Coxeter groups are the dihedral symmetry groups of the equilateral triangle, the square and the regular hexagon. More generally, the Coxeter-Dynkin diagram associated with the dihedral group of symmetries of the regular $m$-gon, for some $m\geq 3$, consists of two vertices that are joined by an edge of weight $a=4\cos^2(\frac{\pi}{m})$. Note that $a\geq 1$. Generalizing the classical construction, the two possible orientations of the diagram yield two possible mutation matrices $B=\pm \bigl(\begin{smallmatrix}0&a\\-1&0\end{smallmatrix}\bigr)$. With this data we associate the following recursion. Let $\mathbf{x}=(x_1,x_2)$ be an initial cluster consisting of positive real numbers $x_1$ and $x_2$. Define a sequence $(x_i)_{i\in\mathbb{Z}}$ of positive real numbers by   
\begin{align}
\label{ExchangeRelation}
x_{i-1}x_{i+1}=\begin{cases}x_i^a+1,&\textrm{if \ } i \textrm{ \ is even;}\\x_i+1,&\textrm{if \ } i \textrm{ \ is odd.}\end{cases} 
\end{align}

In contrast to the cases $m=3,4,6$ the sequences are neither periodic nor do we notice the Laurent phenomenon. But we observe some approximate periodicity: in the case $m=5$ (where we have $a=4\cos^2(\frac{\pi}{5})=\frac12(3+\sqrt{5})\approx2.618033988)$ we have randomly chosen starting values $x_1=0.829497$ and $x_2=0.363532$ from the open interval $(0,1)$, and computed the first few terms numerically, as the first two columns in Figure \ref{Figure:Example} illustrate. After 14 steps, we always get close to our starting values, e.\,g. $x_{-5}\approx x_{9}$ and $x_{-4}\approx x_{10}$. The same phenomenon also occurs for other values of $m$, and the number of steps is either $m+2$ or $ 2(m+2)$ depending on the parity of $m$. The aim of this section is to explain this phenomenon. From now on we assume that $m>4$, because in the other cases we have exact periodicity. Note that $m>4$ implies $a>2$.

\begin{figure}
\begin{center}
\begin{tabular}{|c|c|c|c|}\hline
$n$&$x_n$&$Y_{n/2}$&relative error\\\hline
$-6$&$0.935815$&$0.919721$&$0.017198$\\
$-5$&$0.136311$&&\\
$-4$&$1.214248$&$1.170883$&$0.035714$\\
$-3$&$19.531300$&&\\
$-2$&$16.908654$&$16.788570$&$0.007102$\\
$-1$&$84.093907$&&\\
$0$&$5.032565$&$4.881875$&$0.029943$\\
$1$&$0.829497$&&\\
$2$&$0.363532$&$0.363532$&$0.000000$\\
$3$&$1.290794$&&\\
$4$&$6.301497$&$6.301497$&$0.000000$\\
$5$&$96.739925$&&\\
$6$&$15.510588$&$15.228954$&$0.018158$\\
$7$&$13.546623$&&\\
$8$&$0.937851$&$0.919721$&$0.019332$\\
$9$&$0.136223$&&\\
$10$&$1.211518$&$1.170883$&$0.033541$\\\hline
\end{tabular}
\end{center}
\caption{An example of an approximately periodic sequence $(x_n)$ with $m=5$}
\label{Figure:Example}
\end{figure}

\subsection{A recursion formula for a subsequence}

It is enough to look at every other term of the sequence $(x_i)_{i\in\mathbb{Z}}$, because we can recover every term from the exchange relation (\ref{ExchangeRelation}) once we know its neighbors. To this end, let us define a sequence $(y_i)_{i\in\mathbb{Z}}$ of positive real numbers by putting $y_i=x_{2i}$ for all integers $i\in\mathbb{Z}$. As above, we can view every element $y_i=y_i(x_1,x_2)$ as a function $\mathbb{R}^{+}\times\mathbb{R}^{+}\to\mathbb{R}^{+}$ in the initial values $x_1,x_2$. We will see that the sequence $(y_i)_{i\in\mathbb{Z}}$ is almost periodic; this will imply that the original sequence $(x_i)_{i\in\mathbb{Z}}$ is also almost periodic. The following proposition shows that there is a self-contained recursion formula for the elements of the sequence $(y_i)_{i\in\mathbb{Z}}$. 

\begin{prop} Let $i\in\mathbb{Z}$ be an integer. Then the elements $y_{i-1}$ $y_i$ and $y_{i+1}$ satisfy the equation $y_{i-1}y_iy_{i+1}=y_{i-1}+y_{i+1}+y_i^{a-1}$.
\end{prop}  

\begin{proof} Let $i\in\mathbb{Z}$ be an integer. By construction we have $y_{i-1}=x_{2i-2}$, $y_i=x_{2i}$ and $y_{i+1}=x_{2i+2}$. Note that $y_{i-1}y_i-1=x_{2i-1}$ and $y_iy_{i+1}-1=x_{2i+1}$. The relation $(y_{i-1}y_i-1)(y_iy_{i+1}-1)=y_i^a+1$ yields $y_{i-1}y_iy_{i+1}=y_{i-1}+y_{i+1}+y_i^{a-1}$. 
\end{proof}

\subsection{An approximation of the sequence}

We define another sequence $(Y_i)_{i\in I}$. Moreover, we view the sequence $(Y_i)_{i\in I}$ as a numerical approximation of the sequence $(y_i)_{i\in I}$. The index set $I$ is equal to $I=\{-2,-1,0,1,\ldots,\frac{m}{2}\}$ if $m$ is even and to $I=\{-\frac{m+1}{2},-\frac{m-1}{2},\ldots,\tfrac{m+3}{2},\frac{m+5}{2}\}$ if $m$ is odd. In both cases, we put $Y_1=y_1$ and $Y_2=y_2$, and define the other elements in the sequence recursively. We put: 
\begin{align*}
&Y_0=\frac{Y_2}{Y_1Y_2-1};  &&Y_{-1}=\frac{Y_0^{a-1}}{Y_0Y_1-1}; &&\hspace{-3cm}Y_{-2}=\frac{Y_{-1}^{a-1}}{Y_{0}Y_{-1}}= \frac{Y_{-1}^{a-2}}{Y_{0}};\\
&Y_3=\frac{Y_2^{a-1}}{Y_1Y_2-1}; &&Y_{i+1}=\frac{Y_i^{a-1}}{Y_{i-1}Y_i}= \frac{Y_i^{a-2}}{Y_{i-1}} \textrm{ \ for \ }3\leq i\leq \tfrac{m-1}{2}.
\end{align*}
for all $m$. If $m$ is odd, then we define the missing elements in the sequence as follows:
\begin{align*}
&Y_{(m+3)/2}=\frac{Y_{(m+1)/2}^{a-1}+Y_{(m-1)/2}}{Y_{(m+1)/2}Y_{(m-1)/2}},&&Y_{(m+5)/2}=\frac{Y_{(m+1)/2}}{Y_{(m+3)/2}Y_{(m+1)/2}-1},\\
&Y_{i-1}=\frac{Y_i^{a-1}}{Y_{i+1}Y_i}= \frac{Y_i^{a-2}}{Y_{i+1}} \textrm{ \ \ for \ }-\tfrac{m-3}{2}\leq i\leq-2,&&Y_{-(m+1)/2}=\frac{Y_{-(m-1)/2}^{a-1}+Y_{-(m-3)/2}}{Y_{-(m-1)/2}Y_{-(m-3)/2}}.
\end{align*}

Using the relations $Y_1=x_2$ and $Y_2=x_4=\frac{1+x_1+x_2^a}{x_1x_2}$,  we can view every element $Y_i=Y_i(x_1,x_2)$ as a real-valued function in the initial variables $x_1,x_2$ as above. Note that $x_3=Y_1Y_2-1$. Moreover, $(Y_0Y_1-1)(Y_1Y_2-1)=1$ implies $Y_0Y_1-1=x_3^{-1}$. Thus we have $Y_2=Y_0x_3$.

Next, we will give an explicit formula for the elements of the sequence $(Y_i)_{i\in\mathbb{Z}}$. To do so, we define a sequence $(g_i)_{i\geq 0}$ of integers by $g_0=0$, $g_1=1$ and $g_{i+1}=(a-2)g_i-g_{i-1}$ for $i\geq 1$. The following proposition relates the two sequences.

\begin{prop} 
\label{ExProp}
Let $i\in I$. If $0\leq i\leq\tfrac{m-3}{2}$, then we have $Y_{i+2}=Y_2^{g_i+g_{i+1}}x_3^{-g_i}=Y_0^{g_i+g_{i+1}}x_3^{g_{i+1}}$. Moreover, if $0\leq i\leq \tfrac{m-1}{2}$, then we have $Y_{-i}=Y_2^{g_i+g_{i+1}}x_3^{-g_{i+1}}=Y_0^{g_i+g_{i+1}}x_3^{g_{i}}$ .
\end{prop}

\begin{proof}  The relation $Y_2=Y_0x_3$ implies $Y_2^{g_i+g_{i+1}}x_3^{-g_i}=Y_0^{g_i+g_{i+1}}x_3^{g_{i+1}}$ and $Y_2^{g_i+g_{i+1}}x_3^{-g_{i+1}}=Y_0^{g_i+g_{i+1}}x_3^{g_{i}}$ for all $i$. Trivially, we have $Y_2=Y_2^1x_3^0$ and $Y_0=Y_0^1x_3^{0}$, so that the formulae hold true for $i=0$. By definition, we have $Y_3=Y_2^{a-1}x_3^{-1}$ and $Y_{-1}=Y_0^{a-1}x_3^{1}$ so that the formulae hold true for $i=1$. The general case follows from the definition of the sequence $(Y_i)_{i\in\mathbb{Z}}$ by mathematical induction. 
\end{proof}

Proposition \ref{ExProp} asserts that the sequence $(g_i)_{i\in\mathbb{Z}}$ controls the sequence $(Y_i)_{i\in\mathbb{Z}}$. The following Proposition states the main features of this sequence. We denote by $\omega=\exp(\frac{2\pi i}{m})\in\mathbb{C}$ the root of unity and by $\overline{\omega}\in\mathbb{C}$ its complex conjugate.

\begin{prop} 
\label{PropG}
The sequence $(g_i)_{i\geq 0}$ is periodic. The period is equal to $m$, if $m$ is odd, and equal to $\frac{m}{2}$, if $m$ is even. Moreover, the following formula holds true for all natural numbers $i\geq 0$:
\begin{align}
\label{SeqG}
g_i=\frac{\omega^i-\overline{\omega}^i}{\omega-\overline{\omega}}.
\end{align}
\end{prop}

\begin{proof}
Note that the sequence $(g_i)_{i\in\mathbb{Z}}$ is a homogeneous linear recurrence relation with characteristic polynomial $X^2-(a-2)X+1$. By elementary trigonometry we have $a-2=4\cos^2(\frac{\pi}{m})-2=2\cos(\frac{2\pi}{m})$ so that the characteristic polynomial splits as $(X-\omega)(X-\overline{\omega})$. Therefore, the sequence $(g_i)_{i\in\mathbb{Z}}$ is a $\mathbb{C}$-linear combination of the sequences $(\omega^i)_{i\in\mathbb{N}}$ and $(\overline{\omega}^i)_{i\in\mathbb{N}}$. A comparison of coefficients for the initial values $g_0$ and $g_1$ yields equation $(\ref{SeqG})$. 
\end{proof}

Note that the previous proposition implies that $g_i$ is positive if $1\leq i\leq \tfrac{m-1}{2}$. Similarly, $g_i$ is negative if $-\tfrac{m-1}{2}\leq i\leq -1$. Moreover $\vert g_i\vert\geq1$ unless $i\in\{\tfrac{m}{2},\tfrac{0,m\pm 1}{2}\}$. The explicit formula implies that the following terms in the sequence $(Y_i)_{i\in\mathbb{N}}$ are equal.

\begin{theorem}
\label{Thm:Period}
 If $m$ is even, then the equations $Y_{-2}=Y_{(m-2)/2}$ and $Y_{-1}=Y_{m/2}$ hold. If $m$ is odd, then the equations $Y_{-(m+1)/2}=Y_{(m+3)/2}$ and $Y_{-(m-1)/2}=Y_{(m+5)/2}$ hold.
\end{theorem}

\begin{proof} Let $m$ be even. It is easy to see that $g_{\frac{m}{2}}=0$, $g_{\frac{m}{2}-1}=1$ and $g_{\frac{m}{2}-2}=a-2$, so that $Y_{\frac{m}{2}}=Y_2^{a-1}x_3^{2-a}$, which agrees with $Y_{-1}=Y_0^{a-1}x_3=(Y_2x_3^{-1})^{a-1}x_3=Y_2^{a-1}x_3^{2-a}$. Moreover, it follows that $Y_{\frac{m}{2}-1}=Y_{\frac{m}{2}}^{a-2}Y_2^{-1}x_3=Y_{-1}^{a-2}Y_2^{-1}x_3$, which agrees with $Y_{-2}=Y_{-1}^{a-2}Y_0^{-1}$.

Now let $m$ be odd. Let us put $g=g_{(m-1)/2}$. Due to Proposition \ref{PropG} we have $g_{(m+1)/2}=-g$, from which we conclude $g_{(m-3)/2}=(a-1)g$ and $g_{(m+3)/2}=-(a-1)g$. Furthermore, the recursion implies $g_{(m-5)/2}=(a^2-3a+1)g$. Proposition \ref{ExProp} yields
\begin{align*}
&Y_{-(m-3)/2}=Y_0^{(a-1)g+g}x_3^{(a-1)g}=Y_0^{ag}x_3^{(a-1)g},\\
&Y_{-(m-1)/2}=Y_0^{g-g}x_3^g=x_3^g.
\end{align*}
Using these expressions, we can write the next element of the sequence as
\begin{align*}
Y_{-(m+1)/2}=\frac{x_3^{(a-1)g}+Y_0^{ag}x_3^{(a-1)g}}{Y_0^{ag}x_3^{ag}}=x_3^{-g}(1+Y_0^{-ag}).
\end{align*}
On the other hand, Proposition \ref{ExProp} yields
\begin{align*}
&Y_{(m-1)/2}=Y_0^{(a^2-3a+1)g+(a-1)g}x_3^{(a-1)g}=Y_0^{a(a-2)g}x_3^{(a-1)g},\\
&Y_{(m+1)/2}=Y_0^{(a-1)g+g}x_3^{g}=Y_0^{ag}x_3^g.
\end{align*}
Using these expressions, we can write the next elements of the sequence as
\begin{align}
\label{FormulaY}
&Y_{(m+3)/2}=\frac{Y_0^{a(a-1)g}x_3^{(a-1)g}+Y_0^{a(a-2)g}x_3^{(a-1)g}}{Y_0^{a(a-1)g}x_3^{ag}}=x_3^{-g}(1+Y_0^{-ag}),\\
&Y_{(m+5)/2}=\frac{Y_0^{ag}x_3^{g}}{Y_0^{ag}(1+Y_0^{-ag})-1}=x_3^{g}.\nonumber
\end{align}
The expressions agree with the expressions that we obtain for $Y_{-(m+1)/2}$ and $Y_{-(m-1)/2}$, and so the statement follows.
\end{proof}

\subsection{Numerical comparison of the two sequences}

\begin{theorem} Let $i\in I$. The element $Y_i=Y_i(x_1,x_2)$ is an approximation of $y_i=y_i(x_1,x_2)$ with relative error $\left| \frac{Y_i-y_i}{y_i}\right|=\left| \frac{Y_i(x_1,x_2)-y_i(x_1,x_2)}{y_i(x_1,x_2)}\right|=O(x_1x_2)$ for $x_1,x_2\to 0$. 
\end{theorem}

Before we prove the theorem, we state a lemma. For proofs of the lemma and the theorem recall that our assumption $m>4$ implies $a>2$. Moreover, note that for $x_1,x_2\in (0,1)$ we have
\begin{align*}
&x_3=\tfrac{1+x_2^a}{x_1}>x_1^{-1}>1,&&Y_2=\tfrac{1+x_1+x_2^a}{x_1x_2}>x_1^{-1}x_2^{-1}>1,\\
&Y_2/x_3=\tfrac{1+x_1+x_2^a}{x_2(1+x_2^a)}>x_2^{-1}>1. 
\end{align*}

\begin{lemma} Suppose that $i\in I$ indexes some element of the sequence $(Y_i)_{i\in I}$. \begin{itemize}
\item[(a)] Suppose that $m$ is even and $i\in\{0\}\cup\{3,4,\ldots,\tfrac{m-2}{2}\}$ or $m$ is odd and $i\in\{0,-1,\ldots,-\tfrac{m-3}{2}\}\cup\{3,4,\ldots,\tfrac{m+3}{2}\}$. Then we may write the quotient $(Y_iY_{i-1}-1)/(Y_iY_{i-1})$ as $1-\epsilon_i$ for some the real-valued function $\epsilon_i=\epsilon_i(x_1,x_2)$ with $\vert\epsilon_i(x_1,x_2)\vert=O(x_1x_2)$. 
\item[(b)] If $3\leq i\leq\tfrac{m-1}{3}$, then we may write the quotient $(Y_i^{a-1}+Y_{i-1})/(Y_i^{a-1})$ as $1+\epsilon'_i$ for some real-valued function $\epsilon'_i=\epsilon'_i(x_1,x_2)$ with $\vert\epsilon'_i(x_1,x_2)\vert=O(x_1x_2)$. If $-\tfrac{m-3}{2}\leq i\leq0$, then we may write the quotient $(Y_i^{a-1}+Y_{i+1})/(Y_i^{a-1})$ as $1+\epsilon''_i$ for some real-valued function $\epsilon''_i=\epsilon''_i(x_1,x_2)$ with $\vert\epsilon''_i(x_1,x_2)\vert=O(x_1x_2)$. Finally, if $m$ is odd, then we may write the quotient $(Y_{(m+3)/2}^{a-1}+Y_{(m+1)/2})/Y_{(m+1)/2}$ as $1+\epsilon'''$ for some real-valued function $\epsilon'''=\epsilon'''(x_1,x_2)$ with $\vert\epsilon'''(x_1,x_2)\vert=O(x_1x_2)$. 
\end{itemize}
\end{lemma}

\begin{proof}[Proof of the lemma]
(a) By definition we have $\epsilon_i(x_1,x_2)=(Y_iY_{i-1})^{-1}$. If $3\leq i\leq \tfrac{m+1}{2}$, then Proposition \ref{ExProp} implies 
\begin{align*}
\epsilon_i(x_1,x_2)=(Y_iY_{i-1})^{-1}=(Y_2/x_3)^{-g_{i-2}-g_{i-3}}Y_2^{-g_{i-1}-g_{i-2}}.
\end{align*} 
Thus, it is enough to show that $g_{i-2}+g_{i-3}$ is nonnegative and $g_{i-1}+g_{i-2}$ is at least $1$ for $3\leq i\leq \tfrac{m+1}{2}$, which follows from the explicit formula in Proposition \ref{PropG}. If $m$ is odd and $i=\tfrac{m+3}{2}$, then Proposition \ref{ExProp} and formula (\ref{FormulaY}) imply $Y_iY_{i-1}=1+(Y_2/x_3)^{ag}>(Y_2/x_3)^{ag}$ from which we conclude with the statement of the lemma by similar arguments as above. If $-\tfrac{m-3}{2}\leq i\leq 0$, then Proposition \ref{ExProp} implies 
\begin{align*}
\epsilon_i(x_1,x_2)=(Y_iY_{i-1})^{-1}=(Y_2/x_3)^{-g_{-i+1}-g_{-i+2}}Y_2^{-g_{-i}-g_{-i+1}}.
\end{align*}
Thus, it is enough to show that $g_{i+1}+g_{i+2}$ is nonnegative and $g_{i}+g_{i+1}$ is at least $1$ for $0\leq i\leq\tfrac{m-3}{2}$, which follows from the explicit formula in Proposition \ref{PropG}. 

(b) By definition we have $\epsilon_i'(x_1,x_2)=Y_{i-1}/Y_i^{a-1}$ for all $i$ for which the function is defined. Proposition \ref{ExProp} yields $\epsilon'_i(x_1,x_2)=(Y_2/x_3)^{g_{i-3}-(a-1)g_{i-2}}Y_2^{g_{i-2}-(a-1)g_{i-1}}=(Y_2/x_3)^{-g_{i-1}-g_{i-2}}Y_2^{-g_{i}-g_{i-1}}$. Thus, it is enough to show that $g_{i-1}+g_{i-2}$ and $g_{i}+g_{i-1}$ is greater than $1$ for $3\leq i\leq \tfrac{m-1}{2}$, which follows from the explicit formula in Proposition \ref{PropG}. The cases $\epsilon''_i$ and $\epsilon'''$ are proved similarly.
\end{proof}

With the preparation we are ready to prove the theorem:

\begin{proof}[Proof of the theorem]
The statement is true for $i\in\{1,2\}$ because $Y_1=y_1$ and $Y_2=y_2$ by definition. For $i\in \{3,0\}$ we can estimate the relative errors as $x_1,x_2\to 0$:
\begin{align*}
&\left|\frac{y_3-Y_3}{y_3}\right|=\left|\frac{y_1}{y_1+y_2^{a-1}}\right|=\left|\frac{1}{1+y_1^{-1}y_2^{a-1}}\right|<\left|\frac{1}{1+y_1^{-1}y_2}\right|<\left|\frac{1}{1+x_1^{-1}x_2^{-2}}\right|<x_1x_2^2=O(x_1x_2),\\
&\left|\frac{y_0-Y_0}{y_0}\right|=\left|\frac{y_1^{a-1}}{y_1^{a-1}+y_2}\right|=\left|\frac{1}{1+y_1^{1-a}y_2}\right|<\left|\frac{1}{1+y_1^{-1}y_2}\right|<\left|\frac{1}{1+x_1^{-1}x_2^{-2}}\right|<x_1x_2^2=O(x_1x_2).
\end{align*}

For the other values of $i$ we prove the theorem by induction on the absolute value of $i$. Assume that $i\geq 3$ and that the statement is true for $i$ and $i-1$. We put $Y_i/y_i=1+\delta_i$ and $Y_{i-1}/y_{i-1}=1+\delta_{i-1}$. By induction hypothesis $\vert\delta_i\vert=\vert\delta_i(x_1,x_2)\vert$ and $\vert\delta_{i-1}\vert=\vert\delta_{i-1}(x_1,x_2)\vert$ are real-valued functions in the class $O(x_1x_2)$. Define an auxiliary function
\begin{align*}
\widetilde{Y}_{i+1}=Y_{i+1}(x_1,x_2)=\frac{Y_i^{a-1}+Y_{i-1}}{Y_iY_{i-1}-1}.
\end{align*}
Taylor's theorem for the function $f(x)=(1+x)^{a-1}$ evaluated at $x=\delta_i$ implies that we may write $(1+\delta_i)^{a-1}=1+\delta_i(a-1)(1+\xi)^{a-2}$ for some $\xi$ between $0$ and $\delta_i$. We put $Y_i^{a-1}/y_i^{a-1}=1+\widetilde{\delta}_i$. The previous discussion implies that the function $\vert\widetilde{\delta}_i\vert=\vert\widetilde{\delta}_i(x_1,x_2)\vert$ is a real-valued function in the class $O(x_1x_2)$. The quotient $(Y_i^{a-1}+Y_{i-1})/(y_i^{a-1}+y_{i-1})$ lies between $1+\widetilde{\delta}_i$ and $1+\delta_{i-1}$. Hence we may write the quotient as $1+\delta_i'$ for some the real-valued function $\delta_i'=\delta_i'(x_1,x_2)$ with $\vert\delta_i'(x_1,x_2)\vert=O(x_1x_2)$. Similarly, by the induction hypothesis and the previous lemma we may write the quotient $\tfrac{y_iy_{i-1}-1}{Y_iY_{i-1}-1}=\tfrac{y_iy_{i-1}-1}{Y_iY_{i-1}}\cdot\tfrac{Y_iY_{i-1}}{Y_iY_{i-1}-1}$ as $1+\delta_i''$ for some the real-valued function $\delta_i''=\delta_i''(x_1,x_2)$ with $\vert\delta_i''(x_1,x_2)\vert=O(x_1x_2)$. We conclude that the quotient $\widetilde{Y}_{i+1}/y_{i+1}=1+\delta_i'''$ for some real-valued function $\delta_i'''=\delta_i'''(x_1,x_2)$ with $\vert\delta_i'''(x_1,x_2)\vert=O(x_1x_2)$. 

From the previous lemma we can conclude that we may write the quotient $Y_{i+1}/\widetilde{Y}_{i+1}$ as $1+\widetilde{\epsilon}_i$ for some real-valued function $\widetilde{\epsilon}_i=\widetilde{\epsilon}_i(x_1,x_2)$ with $\vert \widetilde{\epsilon}_i(x_1,x_2)\vert=O(x_1x_2)$. The theorem follows because $Y_i/y_i=(Y_i/\widetilde{Y}_i)\cdot(\widetilde{Y}_i/y_i)=(1+\widetilde{\epsilon}_i)(1+\delta_i''')$ with $\widetilde{\epsilon}_i+\delta_i'''+\widetilde{\epsilon}_i\delta_i'''=O(x_1x_2)$.

The case $i\leq 0$ is proved similarly.
\end{proof}

Thus, the sequence $(y_i)_{i\in I}$ is approximately equal to the sequence $(Y_i)_{i\in I}$ which can be extended to a period sequence $(Y_i)_{i\in\mathbb{Z}}$ in view of Theorem \ref{Thm:Period}. Hence, the original sequence $(x_i)_{i\in\mathbb{Z}}$ is approximately equal to a periodic sequence.

\section{Matrix mutation with real entries}

The matrix mutation rule generalizes to real entries. As before, an $n\times n$ matrix $B=(b_{ij})$ with \textit{real} entries is called \textit{skew-symmetrizable} if there exists a diagonal $n\times n$ matrix $D=\textrm{diag}(d_1,d_2,\ldots,d_n)$ with positive real diagonal entries such that the matrix $DB$ is skew-symmetric, i.\,e. $d_ib_{ij}=-d_jb_{ji}$ for all $1\leq i,j\leq n$. Let $B$ be a real skew-symmetrizable $n\times n$ matrix and $k\in\{1,2,\ldots,n\}$ an index. We define the \textit{mutation} of $B$ at $k$ to be the $n\times n$ matrix $B'=(b_{ij}')$ with entries
\begin{align*}
b'_{ij}=\begin{cases}
-{b_{ij}},&\textrm{if } k\in\{i,j\};\\
{b_{ij}}+\frac{\vert b_{ik}\vert b_{kj}+b_{ik}\vert b_{kj}\vert}2,&\textrm{otherwise}.
\end{cases}
\end{align*} 
As before, we denote the mutation also by $B'=\mu_k(B)$. The following proposition is immediate.

\begin{prop} Let $B$ be a real skew-symmetrizable $n\times n$ matrix with skew-symmetrizer $D$ and let $k\in\{1,2,\ldots,n\}$. The matrix $\mu_k(B)$ is again skew-symmetrizable with skew-symmetrizer $D$. Moreover, $\mu_k^2(B)=B$.
\end{prop}

As before, we define \textit{mutation equivalence} to be the smallest equivalence relation on the set of skew-symmetrizable real $n\times n$ matrices such that $\mu_k(B)\simeq B$ for all $B$ and all $k$.

\begin{rem}
For mutation classes of integer matrices we have several structural results. There is classification of mutation-finite, skew-symmetrizable, integer matrices by Felikson-Shapiro-Tumarkin \cite{FST}. Effective criteria to test whether a given skew-symmetric matrix is mutation-finite are due to Lawson \cite{L} (via minimal mutation-infinite subquivers) and Warkentin \cite{W} (via forks). On the other hand, structural results for classification of mutation-finite, skew-symmetrizable, real matrices seem to be harder, because it is easy to construct finite mutation classes, as the following example shows. 
\end{rem}

\begin{ex}
Let $a,b,c,a',b',c'$ be positive real numbers. Consider the matrix
\begin{align*}
B=\left(\begin{matrix}0&a&-c'\\-a'&0&b\\c&-b'&0\end{matrix}\right).
\end{align*}
Then $B$ is skew symmetrizable if and only if $abc=a'b'c'$. In particular, let $a,b,c\in\mathbb{R}^{+}$ such that $abc=8$. Put $a'=\tfrac{bc}{2}$, $b'=\tfrac{ca}{2}$ and $c'=\tfrac{ab}{2}$. Then $B$ is skew symmetrizable and $\mu_1(B)=\mu_2(B)=\mu_3(B)=-B$. Hence, $B$ is mutation finite.
\end{ex}

The Coxeter group of type $H_3$ is the symmetry group of the regular icosahedron. Besides the Coxeter group of type $H_4$, it is the only noncrystallographic Coxeter group whose rank is greater than $2$. The next lemma shows that the corresponding mutation matrices are mutation finite.

\begin{lemma} Let $a=4\cos^2(\tfrac{\pi}{5})$. The matrices
\begin{align*}
B'=\left(\begin{matrix}0&a&0\\-1&0&1\\0&-1&0\end{matrix}\right),&&B''=\left(\begin{matrix}0&a&0&0\\-1&0&1&0\\0&-1&0&1\\0&0&-1&0\end{matrix}\right)
\end{align*}
of type $H_3$ and $H_4$ are mutation finite. If we identify matrices that are obtained from each other by a simultaneous row and column permutation, then the mutation classes have sizes $16$ and $82$.
\end{lemma}

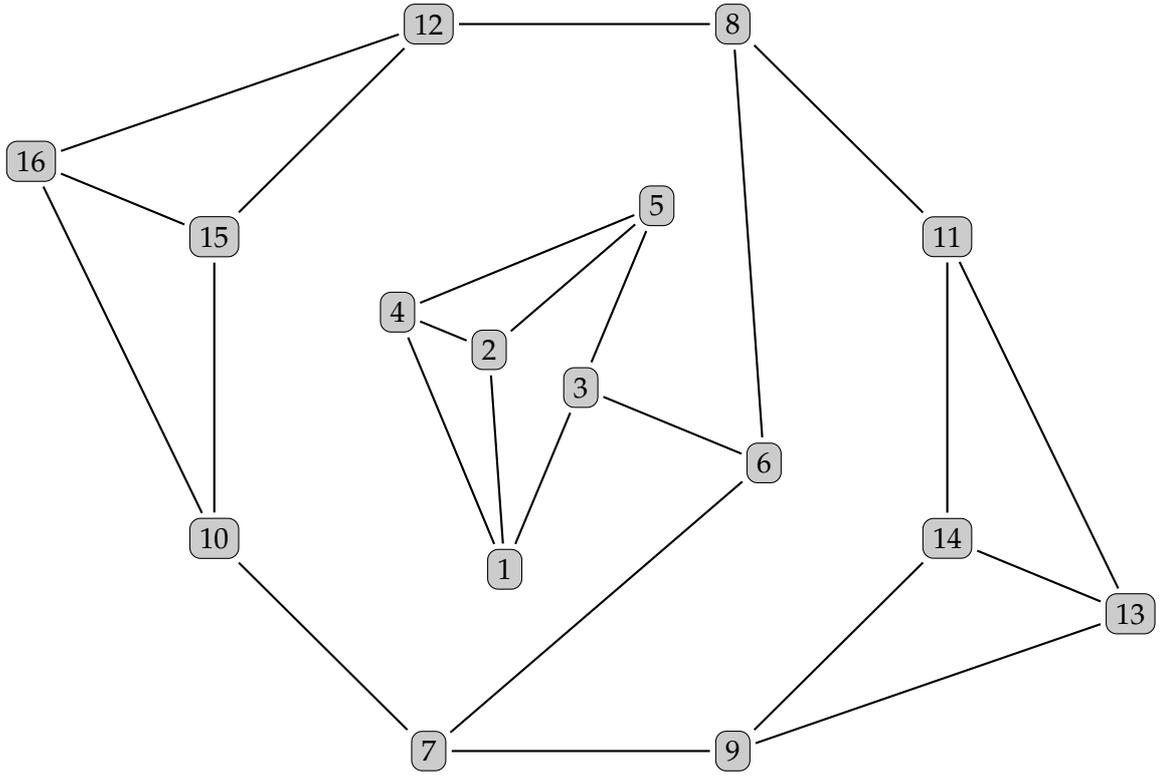
\begin{figure}
\begin{center}
\begin{tikzpicture}
\node[fill=black!20!,rectangle,rounded corners,draw] (7) at (-2,-4.82) {$7$};
\node[fill=black!20!,rectangle,rounded corners,draw] (9) at (2,-4.82) {$9$};
\node[fill=black!20!,rectangle,rounded corners,draw] (14) at (4.82,-2) {$14$};
\node[fill=black!20!,rectangle,rounded corners,draw] (11) at (4.82,2){$11$};
\node[fill=black!20!,rectangle,rounded corners,draw] (8) at (2,4.82) {$8$};
\node[fill=black!20!,rectangle,rounded corners,draw] (12) at (-2,4.82) {$12$};
\node[fill=black!20!,rectangle,rounded corners,draw] (15) at (-4.82,2) {$15$};
\node[fill=black!20!,rectangle,rounded corners,draw] (10) at (-4.82,-2){$10$};
\node[fill=black!20!,rectangle,rounded corners,draw] (16) at (-7.23,3){$16$};
\node[fill=black!20!,rectangle,rounded corners,draw] (13) at (7.23,-3){$13$};
\node[fill=black!20!,rectangle,rounded corners,draw] (6) at (2.41,-1){$6$};
\node[fill=black!20!,rectangle,rounded corners,draw] (2) at (-1.205,0.5){$2$};
\node[fill=black!20!,rectangle,rounded corners,draw] (1) at (-1,-2.41){$1$};
\node[fill=black!20!,rectangle,rounded corners,draw] (5) at (1,2.41){$5$};
\node[fill=black!20!,rectangle,rounded corners,draw] (3) at (0,0){$3$};
\node[fill=black!20!,rectangle,rounded corners,draw] (4) at (-2.41,1){$4$};

\draw[-, thick,shorten >=2pt, shorten <=2pt] (7) to (9);
\draw[-, thick,shorten >=2pt, shorten <=2pt] (9) to (14);
\draw[-, thick,shorten >=2pt, shorten <=2pt] (11) to (14);
\draw[-, thick,shorten >=2pt, shorten <=2pt] (11) to (8);
\draw[-, thick,shorten >=2pt, shorten <=2pt] (12) to (8);
\draw[-, thick,shorten >=2pt, shorten <=2pt] (12) to (15);
\draw[-, thick,shorten >=2pt, shorten <=2pt] (10) to (15);
\draw[-, thick,shorten >=2pt, shorten <=2pt] (10) to (7);
\draw[-, thick,shorten >=2pt, shorten <=2pt] (16) to (12);
\draw[-, thick,shorten >=2pt, shorten <=2pt] (16) to (10);
\draw[-, thick,shorten >=2pt, shorten <=2pt] (16) to (15);
\draw[-, thick,shorten >=2pt, shorten <=2pt] (13) to (9);
\draw[-, thick,shorten >=2pt, shorten <=2pt] (13) to (14);
\draw[-, thick,shorten >=2pt, shorten <=2pt] (13) to (11);
\draw[-, thick,shorten >=2pt, shorten <=2pt] (6) to (8);
\draw[-, thick,shorten >=2pt, shorten <=2pt] (7) to (6);
\draw[-, thick,shorten >=2pt, shorten <=2pt] (3) to (1);
\draw[-, thick,shorten >=2pt, shorten <=2pt] (3) to (5);
\draw[-, thick,shorten >=2pt, shorten <=2pt] (3) to (6);
\draw[-, thick,shorten >=2pt, shorten <=2pt] (2) to (1);
\draw[-, thick,shorten >=2pt, shorten <=2pt] (4) to (1);
\draw[-, thick,shorten >=2pt, shorten <=2pt] (5) to (2);
\draw[-, thick,shorten >=2pt, shorten <=2pt] (4) to (5);
\draw[-, thick,shorten >=2pt, shorten <=2pt] (4) to (2);

\end{tikzpicture}
\end{center}
\label{Figure:Mut}
\caption{The mutation class for $H_3$}
\end{figure}

\begin{proof}
A calculation shows that the following set of $16$ matrices is closed under mutation. Mutations are visualized in the picture. The set contains the matrix $B'$. 
\begin{align*}
&1:\begin{pmatrix}0&1&0\\-1&0&a\\0&-1&0\end{pmatrix},&&7:\begin{pmatrix}0&1-a&a\\a-1&0&1-a\\-1&a-2&0\end{pmatrix},&&12:\begin{pmatrix}0&0&1-a\\0&0&a\\a-2&-1&0\end{pmatrix},\\
&2:\begin{pmatrix}0&-1&0\\1&0&a\\0&-1&0\end{pmatrix},&&8:\begin{pmatrix}0&1-a&a-1\\a-1&0&-a\\2-a&1&0\end{pmatrix},&&13:\begin{pmatrix}0&1-a&0\\a-1&0&a-1\\0&2-a&0\end{pmatrix},\\
&3:\begin{pmatrix}0&-1&a\\1&0&-a\\-1&1&0\end{pmatrix},&&9:\begin{pmatrix}0&a-1&0\\1-a&0&a-1\\0&2-a&0\end{pmatrix},&&14:\begin{pmatrix}0&a-1&0\\1-a&0&1-a\\0&a-2&0\end{pmatrix},\\
&4:\begin{pmatrix}0&1&0\\-1&0&-a\\0&1&0\end{pmatrix},&&10:\begin{pmatrix}0&0&-a\\0&0&a-1\\1&2-a&0\end{pmatrix},&&15:\begin{pmatrix}0&0&a\\0&0&a-1\\-1&2-a&0\end{pmatrix},\\
&5:\begin{pmatrix}0&-1&0\\1&0&-a\\0&1&0\end{pmatrix},&&11:\begin{pmatrix}0&a-1&1-a\\1-a&0&0\\a-2&0&0\end{pmatrix},&&16:\begin{pmatrix}0&0&-a\\0&0&1-a\\1&a-2&0\end{pmatrix}.\\
&6:\begin{pmatrix}0&a-1&-a\\1-a&0&a\\1&-1&0\end{pmatrix},
\end{align*}
A similar argument works in the case $H_4$. 
\end{proof}

\begin{rem}
The following questions which might interesting to investigate in the future: What is a good notion of \textit{cluster algebra} in this context? For example, what is a good choice of an ambient field? Does a sophisticated version of the \textit{Laurent phenomenon} hold? Can we also define approximately periodic sequences for the noncrystallographic cluster algebras of type $H_3$ or $H_4$?
\end{rem}


\begin{thebibliography}{99}
\bibitem{A} Drew Armstrong: \emph{Braid groups, clusters, and free probability: An Outline from the AIM workshop}. Lecture notes available online at \href{http://www.aimath.org/WWN/braidgroups/braidgroups.pdf}{www.aimath.org/WWN/braidgroups/braidgroups.pdf}.
\bibitem{FST} Anna Felikson, Michael Shapiro and Pavel Tumarkin: \emph{Cluster algebras of finite mutation type via unfoldings}. International Mathematics Research Notices \textbf{8}, (2012), 1768--1804. Preprint \href{http://arxiv.org/abs/1006.4276} {arXiv:1006.4276}.
\bibitem{FZ} Sergey Fomin and Andrei Zelevinsky: \emph{Cluster algebras I: Foundations}. Journal of the American Mathematical Society \textbf{15}, no. 2 (2002), 497--529. Preprint \href{http://arxiv.org/abs/math/0104151} {arXiv:math/0104151}.
\bibitem{FZ2} Sergey Fomin and Andrei Zelevinsky: \emph{Cluster algebras II: Finite type classification}. Inventiones Mathematicae \textbf{154}, no. 1 (2003), 63--121. Preprint \href{http://arxiv.org/abs/math/0208229} {arXiv:math/0208229}.
\bibitem{L} John Lawson: \emph{Minimal mutation-infinite quivers}. Preprint \href{http://arxiv.org/abs/1505.01735} {arXiv:1505.01735}.
\bibitem{W} Matthias Warkentin: \emph{Exchange Graphs via Quiver Mutation}. \href{http://www.qucosa.de/fileadmin/data/qucosa/documents/15317/Dissertation_Matthias_Warkentin.pdf}{Dissertation}, Chemnitz (2014).
\end{thebibliography}
\end{document}